\documentclass[preprint]{imsart}
\RequirePackage[OT1]{fontenc}
\RequirePackage{amsthm,amsmath}
\RequirePackage[numbers]{natbib}
\RequirePackage[colorlinks,citecolor=blue,urlcolor=blue]{hyperref}

\issue{}
\firstpage{}
\lastpage{}

\startlocaldefs
\usepackage[tikz]{bclogo}
\usepackage{algorithm2e}
\usepackage[normalem]{ulem}
\usepackage{amssymb, amsmath, natbib, amsthm, graphicx, euscript, mathrsfs, enumerate,textcomp }

\newcommand{\R}{\mathbb{R}}

\newcommand{\eqp}{\stackrel{p}{\longrightarrow}}
\newcommand{\eql}{\stackrel{d}{\longrightarrow}}

\renewcommand{\S}{\mathcal{S}}

\renewcommand{\P}{{\mathbb P}}
\newcommand{\BB}{\mathcal{B}}

\newcommand{\Z}{{\mathbb Z}}
\newcommand{\NN}{{\mathscr N}}

\newcommand{\Tr}{\operatorname{Tr}}

\newcommand{\E}{{\mathbb E}}

\newcommand{\Var}{\operatorname{Var}}

\newcommand{\F}{\mathcal{F}}

\theoremstyle{plain}
\newtheorem{thm}{Theorem}[section]

\newtheorem{rem}[thm]{Remark}
\newtheorem{prop}[thm]{Proposition}
\theoremstyle{definition}

\newcommand{\wS}{\widetilde{\S}}

\newcommand{\M}{M}
\newcommand{\MM}{\mathcal{M}}

\begin{document}
\begin{frontmatter}
\title{Limit Theorems for the Alloy-type Random Energy Model \thanksref{T1}}
\runtitle{Limit theorems for the alloy-type REM}

\begin{aug}
\author{\fnms{Stanislav} \snm{Molchanov}\ead[label=e1]{smolchan@uncc.edu}}
\address{Department of Mathematics and Statistics\\ University of North Carolina at Charlotte\\
Charlotte, NC 28223\\
and \\
International Laboratory of Stochastic Analysis and its Applications\\
National Research University Higher School of Economics \\ Shabolovka, 26, Moscow,  119049 Russia\\
\printead{e1}
}

\author{\fnms{Vladimir} \snm{Panov}\ead[label=e2]{vpanov@hse.ru}}
\address{International Laboratory of Stochastic Analysis and its Applications\\National Research University Higher School of Economics \\ Shabolovka, 26, Moscow,  119049 Russia.\\
\printead{e2}}

\thankstext{T1}{
The study has been funded by the Russian Science Foundation (project  \textnumero 17-11-01098).}
\runauthor{S.Molchanov and V.Panov}

\end{aug}

\begin{abstract}
In this paper, we consider limit laws for the model, which is a generalisation of the  random energy model (REM) to the case when the energy levels have the mixture distribution. More precisely, the distribution of the energy levels is assumed to be a mixture of two normal distributions, one of which is  standard normal, while the second has the mean \(\sqrt{n}a\) with some \(a\in \R,\) and the variance \(\sigma \ne 1\).  
The phase space \((a,\sigma) \subset \R \times \R_+\) is divided onto several domains, where after appropriate normalisation, the partition function converges in law to the stable distribution. These domains are separated by the critical surfaces, corresponding to transitions from the normal distribution to \(\alpha-\)stable with \(\alpha \in (1,2)\), after to 1-stable, and finally to \(\alpha-\)stable with \(\alpha \in (0,1).\) The corresponding phase diagram is the central result of this paper. 
%
%
 
\end{abstract}

\begin{keyword}
\kwd{random energy model}
\kwd{normal mixtures}
\kwd{law of large numbers}
\kwd{central limit theorem}
\kwd{stable distribution}
\kwd{phase transition}
\end{keyword}

\end{frontmatter}


\section{Introduction}
In this paper, we study the limit theorems for the sums of random exponentials 
\begin{eqnarray}
\label{mm}
\S_n (\beta) 
=
\sum_{j=1}^{\lfloor e^n \rfloor}
e^{
    \beta \sqrt{n} Z_j
}, 
\end{eqnarray}
where \(Z_1, Z_2, ... \) are i.i.d. random variables with distribution  equal to 
\begin{eqnarray}\label{F}
F_{a,\sigma}(x) = \frac{1}{2} \Phi\left(x\right) + 
\frac{1}{2} \Phi\biggl(\frac{
x - \sqrt{n} a
}{
\sigma
}
\biggr),
\end{eqnarray}
and by 
 \(\Phi(\cdot)\) we denote the probability distribution function of the standard normal random variable. The sum \(\S_n(\beta)\) is the generalisation of the famous random energy model (REM) introduced by Derrida 
 (\cite{Derrida80}, \cite{Derrida81}) as the simplified version of the Sherrington-Kirkpatrick model of spin glass.  Up to the technical detail (\(\lfloor e^n \rfloor\)  instead of \(2^n\)),  the system in the classical REM is determined by the partition function 
\begin{eqnarray}\label{REM}
S_n(\beta) = \sum_{i=1}^{\lfloor e^n \rfloor } e^{\beta \sqrt{n} \xi_i},
\end{eqnarray}
where \(\xi_i\) are i.i.d random variables with standard normal distribution. Physical interpretation assumes that \(n\) is the size of the system with \(\lfloor e^n \rfloor\) energy levels \(\sqrt{n} \xi_i.\)  On the other side, in this paper we give different interpretation of the model \eqref{mm} in terms of the popular Anderson parabolic problem, see Section~\ref{Anderson}.

Mathematical study of the systems \eqref{mm} and \eqref{REM} is mainly concentrated on  finding the free energy of the model, and on the consideration of  the limit laws depending on the value of the parameter \(\beta. \) As for the free energy for the classical REM~\eqref{REM}, 
Eisele \cite{E}, Olivieri and Picco \cite{OP} show that 
\begin{eqnarray*}
P(\beta) = \lim_{n \to \infty} \frac{
\ln S_n (\beta) 
}{
n
}
=\left\{
\begin{aligned}
1+ \beta^2 /2, & \qquad \beta \leq \sqrt{2}, \\
\sqrt{2} \beta, & \qquad \beta >\sqrt{2}.
\end{aligned}
\right.
\end{eqnarray*} 
In particular, the first line of the r.h.s. 
follows from the fact that for \(\beta\leq \sqrt{2}\) the law of large numbers holds, that is, 
\begin{eqnarray*}
\frac{S_n (\beta)}{\E [S_n(\beta)]} \eqp 1, \qquad n\to \infty,
\end{eqnarray*}
Other limit laws for classical REM model were  proven many years later, in the paper by Bovier, Kurkova, L{\"o}we \cite{BKL}. For instance, it was shown that if \(\beta \leq \sqrt{2}/2,\) then the central limit theorem holds, whereas for \(\beta > \sqrt{2}/2\) the fluctuations of the sum \(S_n(\beta)\)  are non-Gaussian.


Returning to the  model \eqref{mm}, it would be it would be important to note that the free energy for this model was recently studied by Grabchak and Molchanov \cite{GM17}.  The results are based on the observation that the free energy of the whole system is in fact the maximum
\begin{eqnarray}\label{max}
P(\beta) = \max \left\{
  P_1(\beta), P_2(\beta)
\right\},
\end{eqnarray}
where \(P_1\) and \(P_2\) are the free energy functions of the system corresponding to \(\lfloor e^n \rfloor\)  energy levels with \(\NN(0,1)\) distribution, and to \(\lfloor e^n \rfloor\)  energy levels with   \(\NN(\sqrt{n}a,\sigma^2)\) distribution respectively. Nevertheless, similar arguments cannot be applied for proving other limit laws such that the central limit theorem and convergence to the stable distributions. 

In this paper, we aim to show the limiting laws for the model \eqref{mm}. It turns out that  the limiting distributions are stable. The   L\'evy triplet \((\mu, 0, \nu)\) is  such that  \(\mu=\mu(\beta) \in \R_+\) is a drift  and \(\nu\) is a L{\'e}vy  measure on \(\R_+\) defined by
 \[\nu((x,+\infty)) = \frac{
1
}{
\sqrt{2 \pi}	
}
 x^{-\alpha(\beta)}, \qquad \forall \; x>0,\]
 with \(\alpha=\alpha(\beta) \in (0,2].\) In this article we provide the exact forms of the parameters \(\alpha(\beta)\) and \(\mu(\beta)\) and show how these parameters change when the relation between \(a\) and \(\sigma\) varies. Our findings are illustrated by  a phase diagram.   It would be a worth mentioning that the explicit values of the parameters of the limiting stable distribution were not described previously even for more simple model \eqref{REM}, and below we also present the corresponding results for this case.

The rest of the paper is organised as follows. In the next section, we provide some physical motivation of the considered systems in terms of the Anderson parabolic problem.  Next, in Section~\ref{limitlaws}, we give
a new formulation of the results for the sums \eqref{REM}, corresponding to the classical case of standard normal energy levels.   Our main findings related to the case of energy levels with mixture distribution are given in Section~\ref{main}.  All proofs are collected in Section~\ref{proofs}.  For convenience, we also provide a statement of the main results from \cite{GM17} in Appendix~\ref{A}.

\section{Anderson parabolic problem} \label{Anderson}
On the lattice \(\Z^d\), let us consider the cube \(Q_n =\left[-n,n\right]^d\) and the random Anderson Hamiltonian 
\begin{eqnarray}
\label{H}
H_n = \Delta + \beta V_n (x, \omega),
\end{eqnarray}
where \(\beta\) 
 is the reciprocal temperature, \( V_n(x,\omega), x \in Q_n, \) is the random i.i.d. potential (on some probability space, "environment", \(\left(\Omega, \F, \P \right)\)), and 
\[\Delta \psi(x) = \sum_{x': |x'-x|=1} \psi (x')\]
is the lattice Laplacian on \(Q_n\) with the Dirichlet boundary condition \[\psi(x) = 0,\qquad x \in \partial Q_n.\]
 We assume that potential is "very strong": \(V_n(x, \omega) = \sqrt{n} \xi(x,\omega),\) where \(\xi(x, \omega)\) are i.i.d. r.v.'s with the law \eqref{H}. Here the factor $\sqrt{n}$ is related to the Gaussian law. If, say, \(\P \left\{ \xi>a \right\} =\exp \left\{ -a^\alpha / \alpha\right\}, \; \alpha>1\) (Weibull's law), like in \cite{benarous}, instead of \(\sqrt{n} \) one have to use \(n^{1/\alpha}.\)
 
 Consider the parabolic problem 
 \begin{eqnarray*}
\frac{
\partial u
}{
\partial t
}
&=& H_n u, \qquad t \geq 0, \; x \in Q_n, \\  u(t,x) &=& 0, \qquad \qquad \qquad x \in \partial Q_n,\\
u(0,x) &=& \delta_y(x),	
\end{eqnarray*}
where \(y \in  Q_n \) is considered as a parameter.
Then the fundamental solution of this problem is given by 
\begin{eqnarray}\label{uuu}
u_n(t,x) = 
u_n(t,x,y)  = 
\sum_{i=1}^{|Q_n|}
e^{
\lambda_{n,i} t
}
\psi_{n,i}(x) \psi_{n,i}(y),
\end{eqnarray}
where \(\lambda_{n,i}, \psi_{n,i}\) are the eigenvalues and (normalised) eigenfunctions of the operator \(H_n\), that is, \(H_n \psi_{n,i} = \lambda_{n,i} \psi_{n,i}.\) Now 
\begin{eqnarray}\label{tr}
\Tr e^{t H_n} = \sum_{x \in Q_n} u_n(t,x,x) = \sum_{i=1}^{|Q_n|}  e^{t \lambda_{n,i}}
\end{eqnarray}
is the random exponential sum. The asymptotic analysis  of this sum and related concepts of the intermittency and localisation were discussed in numerous works, e.g.,  \cite{BMR2}, \cite{BMR},
\cite{den2012random}, \cite{gartner2007geometric},
\cite{gartner1990parabolic}, \cite{gartner1998parabolic}. 

The limit theorems for the parabolic Anderson model become the subject of the studies only recently and the picture here is still not complete. In comparison with the setup discussed in \cite{benarous} and \cite{BKL}, the main difficulty is the dependence of \(\lambda_{n,i} (\omega)\).  However, for the "very strong" potentials the situation is simplier. In this case, we can naturally use the parameter \(y\) instead of \(i,\) because for any \(y\) there exists the eigenfunction equal to the \(\delta\)-function. More precisely, if \(V_n(x, \omega) = \sqrt{n} \xi(x,\omega), \; x\in Q_n\), then with high accurancy \[\psi_{n,y}(x) =\delta_y(x), \qquad \lambda_{n,y}(\omega)= \sqrt{n} \xi(y, \omega).\]
One can estimate the errors using usual perturbation arguments, but we will not provide the caculations here. Let us simply state that at the level of physical intuition the sum 
\begin{eqnarray*}
\sum_{y \in Q_n} e^{\beta \sqrt{n} \xi(y, \omega)} = S_n(\beta) 
\end{eqnarray*}
is close to sum \eqref{tr} for \(t=\beta.\)

In this setting, the mixtures appear naturally. Assume that we have two highly disordered Hamiltonians with potentials \begin{eqnarray*}
V_1 (x, \omega) = \sqrt{n} \eta (x,\omega),
\quad 
V_2 (x, \omega) = \sqrt{n} \zeta (x,\omega),
\qquad x \in Q_n,
\end{eqnarray*}
where \(\eta, \zeta \) are two independent systems of i.i.d. Gaussian r.v.'s with different parameters. Let us consider the new alloy-type potential 
\begin{eqnarray*}
V_n (x, \omega) = \sqrt{n} \xi (x,\omega),
\qquad \mbox{where} \quad \xi (x, \omega) = 
\begin{cases} 
\eta(x, \omega) & \mbox{with probability } 1/2,\\
\zeta(x, \omega) & \mbox{with probability } 1/2.
\end{cases}
\end{eqnarray*}
The trace of the parabolic Anderson problem with this potential is close to the sum \eqref{mm}.

\section{Limit laws for the sums of normal exponentials}
\label{limitlaws}
Before we will present our results for the model \eqref{mm}, we would like to shortly discuss the corresponding results for the classical REM model.

The limit laws for the sums \eqref{REM} with standard normal r.v. \(\xi_i\)  were firstly shown in \cite{BKL} and later generalised in \cite{benarous}. Below we formulate Proposition~\ref{prop1}, which can be considered as a new version of the results from \cite{BKL}. The main advantage of our version  in comparison with the previously known facts (given in \cite{BKL} and \cite{KK}) is that we provide the exact form of the limiting distribution. As follows from the next proposition, this distribution is in fact from the class of stable laws.


\begin{prop}\label{prop1}
\begin{enumerate}[(i)]
\item If \(\beta< \sqrt{2},\) then the law of large numbers holds, that is, 
\begin{eqnarray*}
\frac{S_n (\beta)}{\E [S_n(\beta)]} \eqp 1, \qquad n\to \infty.
\end{eqnarray*}
 If \(\beta =  \sqrt{2},\) then 
\begin{eqnarray*}
\frac{S_n (\beta)}{\E [S_n(\beta)]} \eqp 1/2, \qquad n\to \infty.
\end{eqnarray*}
\item If \(\beta<\sqrt{2}/2,\) then the central limit theorem holds, that is, 
\begin{eqnarray*}
\frac{ 
S_{n} (\beta)  - \E [S_{n} (\beta)]
}
{
\sqrt{\Var (S_n(\beta))}
}
	\eql \NN(0,1), \qquad n \to \infty.
\end{eqnarray*}
If \(\beta = \sqrt{2}/2\), then 
\begin{eqnarray*}
\frac{ 
S_{n} (\beta)  - \E [S_{n} (\beta)]
}
{
\sqrt{\Var (S_n(\beta))}
}
	\eql \NN(0,1/2), \qquad n \to \infty.
\end{eqnarray*}

\item For \(\beta>\sqrt{2}/2,\) it holds
\begin{eqnarray*}
	\frac{S_{n}(\beta)-\delta_n(\beta)}{\gamma_n(\beta)}  \eql F_{\alpha(\beta), \mu(\beta)}, \qquad n \to \infty,
\end{eqnarray*}
where  \(F_{\alpha(\beta), \mu(\beta)}\) stands for  a stable distribution on \(\R_{+}\), that is,  an infinitely divisible distribution with L\'evy triplet \((\mu, 0, \nu)\) such that  \(\mu=\mu(\beta) \in \R_+\) is a drift  and \(\nu\) is a L{\'e}vy  measure on \(\R_+\) defined by
 \[\nu((x,+\infty)) = \frac{
1
}{
\sqrt{2 \pi}	
}
 x^{-\alpha(\beta)}, \qquad \forall \; x>0.\]
Moreover,
\begin{eqnarray*}
\alpha(\beta) &=& \sqrt{2}/\beta,\\
\gamma_n (\beta) &=& 
(2n)^{- \beta /(2 \sqrt{2})}
e^{\sqrt{2} \beta n},		
\end{eqnarray*}
and the choices of \(\mu (\beta)\) and \(\delta_n(\beta)\) 
are related to each other; for instance, they can be taken as follows: 
\begin{eqnarray*}
\left(\mu (\beta), \delta_n(\beta) \right)&=& \begin{cases}
\bigl( 0, \E[S_n(\beta)] \bigr), &\text{if \(\beta \in (\sqrt{2}/2, \sqrt{2})\)},\\
\bigl( 0, \frac{1}{2}\E[S_n(\beta)] \bigr), &\text{if \(\beta = \sqrt{2}\)},\\
\bigl( \frac{1}{ \sqrt{\pi} \left(\beta  - \sqrt{2}\right)},0
\bigr)
, &\text{if \(
\beta \in (\sqrt{2}, +\infty)
\).}
	\end{cases}\end{eqnarray*}
	\end{enumerate}
\end{prop}
\begin{proof} We provide the proof of this statement in Section~\ref{proof1}. 
\end{proof}

%

 \section{Main results}\label{main}
Now let us return to the model \eqref{mm}, which can be rewritten as 
\begin{eqnarray*}
\S_n (\beta) 
=
\sum_{j=1}^{\nu_n^1}
\exp\left\{
    \beta \sqrt{n} \xi_j^1
\right\}
+
\sum_{j=1}^{\nu_n^2}
\exp\left\{
    \beta \left( 
     \sigma \sqrt{n} \xi_j^2
    + a n
    \right)
\right\}:= \S_n^1 (\beta) + \S_n^2 (\beta), 
\end{eqnarray*}
 where \(\xi_j^1, \xi_j^2\) are 2 independent sequences of i.i.d. standard normal r.v.'s, \(\nu_n^1, \nu_n^2\) have binomial distribution with parameters \((\lfloor e^n \rfloor, 0.5)\), and satisfy \(\nu_n^1 + \nu_n^2 = \lfloor e^n \rfloor.\)


The next 3 theorems yield the values of \(\beta,\)  for which  the law of large numbers, the central limit theorem and the convergence to stable distributions hold. 
\begin{thm}[Law of large numbers]\label{lln}
It holds
\begin{eqnarray}\label{a1}
	\frac{\S_{n}(\beta)}{\E [\S_n(\beta)]} \eqp 1, \qquad n\to \infty,
\end{eqnarray}
provided that
\(\beta<\beta^+\), where
\begin{eqnarray}\label{bb}
\beta^+ =
\begin{cases}
\sqrt{2}/\sigma, &\text{if \(a>\left(
  1-\sigma^2
\right) /( \sqrt{2} \sigma\)),}\\
\beta_\circ:= \frac{
2 a
}{ 1 - \sigma^2
}
,  &\text{if \(
 \left(
  1-\sigma^2
\right) / \sqrt{2} <a <
\left(
  1-\sigma^2
\right) /( \sqrt{2} \sigma)
\)},\\
\sqrt{2}, &\text{if \(a<\left(
  1-\sigma^2
\right) / \sqrt{2}\)}.\\
	\end{cases}
\end{eqnarray}
\end{thm}

\begin{thm}[Central limit theorem]\label{clt} It holds
\begin{eqnarray}\label{a2}
\frac{ 
\S_{n} (\beta)  - \E [\S_{n} (\beta)]
}
{
\sqrt{\Var (\S_n(\beta))}
}
	\eql \NN(0,1), \qquad n \to \infty,
\end{eqnarray}
provided that
\(\beta<\beta^+ / 2 \), where
\(\beta^+\) is defined by \eqref{bb}. 
\end{thm}
\begin{thm}[Convergence to the stable distribution]\label{stable}
\hspace{2cm}
\begin{enumerate}[(i)]
\item 
If \(a<\sqrt{2} (1-\sigma)\), then there exists a deterministic sequence \(a_n^\sharp(\beta)\) such that 
\begin{eqnarray*}
	\frac{S_{n}(\beta)-a_n^\sharp(\beta)}{\gamma_n (\beta)}   \eql F_{\sqrt{2}/\beta, 0}, \qquad n \to \infty,
\end{eqnarray*}
for any \(\beta>\beta^\sharp\), where 
\begin{eqnarray*}
\beta^\sharp&=&
\begin{cases}
\beta_\diamond/2, &\text{if \(\sigma<1\) and \(a>(1-\sigma^2) / \sqrt{2}\),}\\
\sqrt{2} / 2, &\text{otherwise,}
\end{cases}
\end{eqnarray*}
with \(\beta_\diamond = \left(
 ( \sqrt{2} -a ) - \sqrt{\left(
  \sqrt{2} - a 
\right)^2  - 2 \sigma^2}
\right)/\sigma^2.\)
\item If \(a>\sqrt{2} (1-\sigma)\), then there exists a deterministic sequence \(\breve{a}_n(\beta)\) such that 
\begin{eqnarray*}
	\frac{S_{n}(\beta)-\breve{a}_n(\beta)}{e^{\beta a n} \gamma_n (\beta \sigma)}   \eql F_{\sqrt{2}/(\beta\sigma), 0}, \qquad n \to \infty,
\end{eqnarray*}
for any \(\beta>\breve{\beta}\),where 
\begin{eqnarray*}
\beta^\sharp&=&
\begin{cases}
\beta_*/2, &\text{if \(\sigma>1\) and \(a<(1-\sigma^2) / (\sqrt{2} \sigma)\),}\\
\sqrt{2} / (2\sigma), &\text{otherwise,}
\end{cases}
\end{eqnarray*}
with \(\beta_{*} = \left(
  \sigma \sqrt{2} +a
\right) - \sqrt{\left(
  \sigma \sqrt{2} +a
\right)^2 - 2}\).
\end{enumerate}



\end{thm}
Theorems~\ref{lln}, \ref{clt}, \ref{stable} basically mean that there exist 6 essentially different types of relation between  \(a\) and \(\sigma)\). Figure~\ref{fig1} illustrates the division of the area \((a,\sigma) \in \R \times \R_+\)  into 6 subareas with different asymptotic behaviour of the sums \(\S_n(\beta).\)

\begin{figure}
\begin{center}
\includegraphics[width=0.8 \linewidth ]{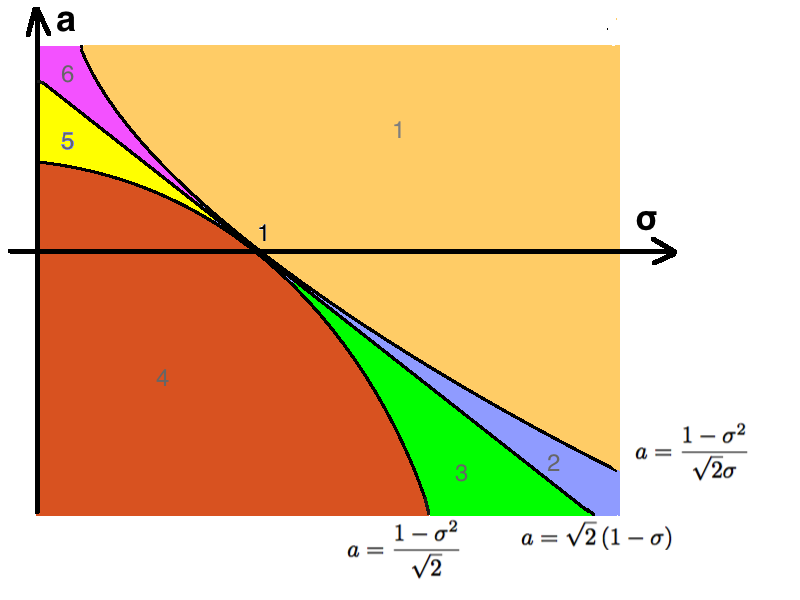}\caption{\label{fig1}Illustration of the  asymptotic behavior of 
the sums \(\S_n\) depending on \(a\) and \(\sigma\):\newline
- Zone 1(orange),   \(a>\left(
  1-\sigma^2
\right) /( \sqrt{2} \sigma)\):  CLT holds for any \(\beta< \sqrt{2}/(2 \sigma)\) whereas the sum converges to stable distribution for any \(\beta>\sqrt{2}/(2\sigma);\) LLN holds for \(\beta< \sqrt{2}\sigma.\) \newline
- Zone 4 (red), \(a<\left(
  1-\sigma^2
\right) / \sqrt{2}\): 
CLT holds for any \(\beta<\sqrt{2}/2,\) 
whereas the sum converges to stable distribution for any \(\beta>\sqrt{2}/2;\) 
LLN holds for \(\beta< \sqrt{2}.\)\newline
- Other zones, \(\left(
  1-\sigma^2
\right) / \sqrt{2}< a<\left(
  1-\sigma^2
\right) /( \sqrt{2} \sigma)\): 
CLT holds for any \(\beta<\beta_\circ/2,\) LLN - for \(\beta< \beta_\circ, \) and the fluctuations are stable if \newline 
\(\vartriangleleft\) Zone 2 (blue):  \(\sigma>1\) and
\( \sqrt{2} (1-\sigma) <a <
\left(
  1-\sigma^2
\right) /( \sqrt{2} \sigma)
\): \quad
\(\beta > \beta_*/2\); \newline
\(\vartriangleleft\) Zone 3 (green):   \(\sigma>1\) and \(
\left(
  1-\sigma^2
\right) / \sqrt{2} 
 <a < \sqrt{2} (1-\sigma)
\): \quad 
\(\beta > \sqrt{2}/2\); \newline
\(\vartriangleleft\) Zone 5 (yellow): \(\sigma<1\) and \(
\left(
  1-\sigma^2
\right) / \sqrt{2} 
 <a < \sqrt{2} (1-\sigma)
\): \quad 
\(\beta > \beta_\diamond/2\); \newline
\(\vartriangleleft\) Zone 6 (purple):   \(\sigma<1\) and  \(
 \sqrt{2} (1-\sigma) <a <
\left(
  1-\sigma^2
\right) /( \sqrt{2} \sigma)
\): \quad 
\(\beta > \sqrt{2}/(2\sigma)\).
} 
\end{center}\end{figure}

\begin{rem}
Analogously to  Proposition~\ref{prop1},  for the critical values \(\beta=\beta^+\) and \(\beta=\beta^+/2\) the convergence in \eqref{a1} and \eqref{a2} holds, but the limits are equal to \(1/2\) and \(\NN(0,1/2)\) resp.  \end{rem}


\section{Proofs}\label{proofs}
\subsection{Proof of Proposition~\ref{prop1}}
\label{proof1}
The proof is based on the Proposition~3.1 from \cite{Panov2017}, which is in fact a combination of  Theorem~1.7.3 from \cite{il}, Theorem~3.2.2 from \cite{ms}, and a number of theorems given in Chapter~IV from \cite{petrov}. Below we provide the proof for (iii), because other parts of this proposition were shown in previous papers.

\textbf{1. Choice of  \(\gamma_n\).} \label{ppp} First, we find a sequence \(\gamma_n\) such that the sum 
\begin{eqnarray*}
\Sigma_n (x) &:=& \sum_{i=1}^{\lfloor e^n \rfloor} \P\left\{
e^{
    \beta \sqrt{n} \xi_i
}> \gamma_n x\right\} 
\end{eqnarray*} has non-trivial limit as \(n\) tends to infinity.  We get 
\begin{eqnarray*}
\Sigma_n (x)
=
\lfloor e^n \rfloor \cdot \P\left\{
\xi_1 > \kappa_n(x)\right\} 
=
\lfloor e^n \rfloor 
\cdot
\left(
  1 - \Phi\left(
\kappa_n(x)\right)
\right),
\end{eqnarray*}
where \(\Phi(x)\) stands for  the distribution of the standard normal random variable and \(
 \kappa_n(x):= \log(\gamma_n x) / 
    \left( \beta \sqrt{n} \right)\).
Next, taking into account that 
\begin{eqnarray}\label{Phi}
\Phi(x) = 1- (\sqrt{2\pi} x)^{-1} e^{-x^2/2} (1+o(1)), \qquad \mbox{as} \quad x \to \infty,\end{eqnarray}we conclude that 
\begin{eqnarray*}\label{SS}
\Sigma_n (x)
&=&
\frac{1}
{ \sqrt{2 \pi}}
\frac{\lfloor e^n \rfloor }{\kappa_n(x)}
e^{-\kappa^2_n(x)/2}\\
&=&
\frac{1}{\sqrt{2 \pi}}
\frac{\lfloor e^n \rfloor }
{\frac{\log\gamma_n}{\beta \sqrt{n}} +
\frac{\log(x)}{\beta \sqrt{n}}
}
\cdot e^{
-\frac{1}{2}
\left(
 \frac{\log(\gamma_n)}{\beta \sqrt{n}} 
\right)^2
-
  \frac{\log(\gamma_n)}{\beta \sqrt{n}} 
\frac{
\log(x)
}{
\beta \sqrt{n}
}
-\frac{1}{2} 
\left(
   \frac{\log(x)}{\beta \sqrt{n}} 
\right)^2
}.
\end{eqnarray*}
Let us find \(\gamma_n\) in the form 
\begin{eqnarray*}
  \frac{\log(\gamma_n )}{\beta \sqrt{n}} 
= C \sqrt{n} + g(n),
\end{eqnarray*}
where \(g(n) = o(\sqrt{n}).\) 
 We get 
\begin{eqnarray*}
\Sigma_n (x)=
\frac{1}{ \sqrt{2 \pi}}
\frac{\lfloor e^n \rfloor }
{C \sqrt{n} + g(n) + 
\frac{\log(x)}{\beta \sqrt{n}}
}
\cdot e^{
-\frac{1}{2}
\left(
 C \sqrt{n} + g(n)
\right)^2
-
\frac{
C
}{
\beta 
}\log(x)
}(1 + o(1)).
\end{eqnarray*}
Therefore,   under the choice \(C=\sqrt{2},\) \(g(n)=- 
\log (\sqrt{2n})
/
\sqrt{2n},\)
the sum \(\Sigma_n\) converges to a non-trivial limit, namely,
\begin{eqnarray*}
\Sigma_n (x) \to \frac{1}{ \sqrt{2 \pi}}
x^{-\sqrt{2} / \beta }, \qquad \mbox{as} \quad n \to \infty.
\end{eqnarray*}
Therefore, we conclude that  the choice 
\begin{eqnarray*}
\gamma_n=
(2n)^{-\beta /(2 \sqrt{2})}
e^{\sqrt{2} \beta n}
\end{eqnarray*}
yields the convergence of \(S_n/\gamma_n\) to a non-trivial limit.

\textbf{2. Condition on the truncated moments.} Let us  analyse the asymptotic behaviour of the expression 
\begin{eqnarray}\label{Jn}
J_n(s, \tau) := \lfloor e^n \rfloor  \int_0^{\tau} x^{s} \mu_{n} (dx),
\end{eqnarray} where \(\mu_{n}\) is the distribution of \(e^{\beta \sqrt{n} \xi_1} / \gamma_n\), \(\tau>0\) and \(s = \{1,2\}.\)  
Our choice of \(\gamma_n\) 
yields 
\begin{eqnarray}
\nonumber
J_n (s,\tau) 
&=&
\frac{\lfloor e^n \rfloor }{\beta \sqrt{n}}  \int_0^{\tau} x^{s-1}
p\left(
  \frac{
\log(\gamma_n x)
}{
\beta \sqrt{n} 
}
\right)
dx\\
\nonumber
 &=& 
\frac{
\lfloor e^n \rfloor
}{
\gamma_n^{s}
}
\int_{-\infty}^{\frac{
\log(\gamma_n \tau)
}{
\beta \sqrt{n}
}
}
e^{\beta \sqrt{n} s y} p\left(
y\right) dy\\
\nonumber
&=& 
\frac{
\lfloor e^n \rfloor
e^{\beta^2 n s^2/2}
}{
\gamma_n^{s}
}
\int_{-\infty}^{
\frac{
\log(\gamma_n \tau)
}{
\beta \sqrt{n}
}
}
\frac{
1
}{
\sqrt{2 \pi}
}
e^{-(y-\beta \sqrt{n} s)^2 /2}
dy\\
&=&
\label{jn}
\frac{
\lfloor e^n \rfloor
e^{\beta^2 n s^2/2}
}{
\gamma_n^{s}
}
\Phi\left(
  \frac{
\log(\gamma_n \tau)
}{
\beta \sqrt{n}
}
-\beta s  \sqrt{n}
\right),
\end{eqnarray}
where \(p\) is the density of a standard normal random variable. 
Since \(\log(\gamma_n) \asymp \sqrt{2}  \beta n,\)
 we get that 
 \begin{eqnarray*}
  \frac{
\log(\gamma_n \tau)
}{
\beta \sqrt{n}
}
-\beta s  \sqrt{n}
= \left(
  \sqrt{2} - \beta s +
  \frac{
\log(\tau)
}{
\beta n
}
\right) \sqrt{n} - \frac{
1
}{
2\sqrt{2}
}
\frac{
\log(2n)
}{
\sqrt{n}
}.
\end{eqnarray*}
Finally, applying \eqref{Phi}, we conclude that as \(n \to \infty\)
\begin{eqnarray} \label{Jn}
 J_n (s, \tau) \asymp \begin{cases}
\left(
\sqrt{\pi} \left(\beta s - \sqrt{2}\right)
\right)^{-1} \tau^{\beta s  - \sqrt{2}}
, &\mbox{if \(\sqrt{2} - \beta s <0\),}\\
\lfloor e^n \rfloor
\cdot
\E \left[e^{\beta s \sqrt{n} \xi_1}\right]
\gamma_n^{-s}, &\mbox{if \(\sqrt{2} - \beta s >0\)},\\
\frac{
1
}{2
}
\lfloor e^n \rfloor
\cdot
\E \left[e^{\beta s \sqrt{n} \xi_1}\right]
\gamma_n^{-s}, &\mbox{if \(\sqrt{2} - \beta s =0\)}
\end{cases}
\end{eqnarray}
where \(\E \left[e^{\beta s \sqrt{n} \xi_1}\right] = e^{\beta^2 n s^2 /2},\) because 
of the fact that 
\begin{eqnarray}\label{Ee}
\E e^{c \xi} = \frac{
1
}{
\sqrt{2 \pi}
}
\int_\R e^{cx}  e^{-x^2 / 2} dx
=
 \frac{
1
}{
\sqrt{2 \pi}
}
\int_\R   e^{-(x-c)^2 / 2} dx \cdot e^{c^2 /2 }
=
e^{c^2 /2 }.
\end{eqnarray}

Therefore 
\begin{eqnarray*}
\lim_{\tau \to 0} \lim_{n \to \infty} \left(
  J_n (2, \tau) - \frac{
1
}{\lfloor e^n \rfloor
}
\left(
  J_n (1, \tau)
\right)^2
\right) =0,
\end{eqnarray*}
and therefore the condition (14) from \cite{Panov2017} 
is fulfilled for any \(\beta >\sqrt{2}/2.\) To conlcude the proof, it is sufficient to note that 
\(a_n = J_n(1,1)+o(1),\)  and its asymptotical behaviour follows from \eqref{Jn}.


\subsection{Proof of Theorem~\ref{lln}}
\textbf{1.} Denote 
\begin{eqnarray*}
\wS_{n} (\beta)=\frac{\S_{n}(\beta)}{\E[\S_{n}(\beta)]} -1 =
 \frac{
 	\sum_{j=1}^{\M}
	\left(
	       e^{\beta \sqrt{n}Z_j} - \E 	  e^{\beta \sqrt{n}Z_j}
	\right)
}{
	\M \cdot  \E 
		  e^{\beta \sqrt{n}Z_1}
},
\end{eqnarray*}
where \(M:=\lfloor e^n \rfloor.\)
Our aim is to show that there exists a constant \(r>1\) such that \(\E
	| 
		\wS_{n}(\beta)
	|^{r} 
 \to 0\) as \(n \to \infty\). This will imply that \(\wS_{n}(\beta) \eqp 0,\) and therefore the result will follow.

Applying the Bahr-Esseen inequality for \(r \in(1,2),\)  see \cite{bahresseen}, we get that 
\begin{eqnarray}
\E\left[
	| 
		\wS_{n}(\beta)
	|^{r} 
\right]
\leq
C_r
\frac{
 	\sum_{j=1}^{M}
	\E
	\left[
\left|			         e^{\beta \sqrt{n}Z_j} - \E 		  e^{\beta \sqrt{n}Z_j}
\right|^{r}
	\right]
}{
	\left( M \cdot
	\E 		  e^{\beta \sqrt{n}Z_1}
	\right)^{r}
}
=C_r M^{1-r} \frac{\MM_{n}(r)}{\left( \E 		  e^{\beta \sqrt{n}Z_1}
	\right)^{r}},
\label{BE}\end{eqnarray}
where \(C_r\) is  some constant depending on \(r\), and \(\MM_{n}(r):=\E
|			         e^{\beta \sqrt{n}Z_1} - \E e^{\beta \sqrt{n}Z_1}
|^{r}
\). The further analysis consists in establishing the asymptotical behavior of the numerator and denominator of the last fraction in \eqref{BE}.

\textbf{2.} Note that 
 for any \(c \in \R,\)
\begin{eqnarray*}
\E e^{c Z_1}&=&
\frac{
1
}{
2
} \E e^{c \xi} + \frac{
1
}{
2
} \cdot  \E e^{c (
  \sqrt{ n} a +  \sigma \xi
)}\\
&=&
\frac{
1
}{
2
}  \cdot e^{c^2 /2 } + \frac{
1
}{
2
} \cdot  e^{c 
  \sqrt{ n} a +  c^2 \sigma^2/2 },
\end{eqnarray*}
where we use \eqref{Ee}.  Therefore, the denominator in \eqref{BE} is equal to the \(r-\) th power of
\begin{eqnarray}\label{exp}
\E e^{\beta \sqrt{n}Z_1}&=&
\frac{
1
}{
2
}e^{(\beta^2 / 2) n} + \frac{
1
}{
2
} \cdot e^{(\beta a + \beta^2  \sigma^2 /2) n }\\
&=& \kappa e^{\gamma(\beta) n} \left(
  1 + o(1)
\right), \quad n \to \infty.
\nonumber
\end{eqnarray}

\textbf{3.} Next, we proceed with studying the asymptotics of \(\MM_n(r).\) Taking into account \eqref{Ee}, we get
\begin{eqnarray}\nonumber
\MM_n(r) 
&\leq& 2^r 
\left(
 \E e^{r \beta \sqrt{n}Z_1}   
 + 
 \left(
   \E e^{\beta \sqrt{n}Z_1}
\right)^r
\right)
\\
\nonumber
&\leq& 
2^r \left[
\frac{
1
}{
2
} \cdot e^{(r^2 \beta^2 / 2) n} + \frac{
1
}{
2
}\cdot e^{r\beta a n+ (r^2 \beta^2  \sigma^2 /2) n }
\right.\\
\nonumber
&&
\hspace{1cm}
\left. 
+\left( 
\frac{
1
}{
2
}  e^{( \beta^2 / 2) n} + \frac{
1
}{
2
} e^{\beta a n+ (\beta^2  \sigma^2 /2) n }
\right)^r
\right]\\
\label{MM}
&=& 2^r \left(
\frac{
1
}{
2
} e^{(r^2 \beta^2 / 2) n} + \frac{
1
}{
2
} e^{( r\beta a + r^2 \beta^2  \sigma^2 /2) n }
\right) (1 +o(1)).
\end{eqnarray}

\textbf{4.} Finally, returning to \eqref{BE}, we conclude that
\begin{eqnarray}\label{final}
\E\left[
 | \wS_{n}(\beta)|^{r} 
\right]
 \leq
G_r \cdot 
\exp\Bigl\{ 
H_r (\beta) n
\Bigr\}
\left(
  1 + o(1) 
\right),\end{eqnarray}
where \(G_r>0\) is a constant depending on \(r\),
\begin{eqnarray*}
H_r(\beta) :=\left(
1-r
\right) + r\left( 
 \lambda_r(\beta)
-  \gamma(\beta) 
\right),
\end{eqnarray*}
 and
\begin{eqnarray*}
 \lambda_r(\beta) &=& \max( r \beta^2 / 2, \beta a + r \beta^2  \sigma^2 /2),\\
\gamma(\beta) &=& \max( \beta^2 / 2, \beta a +  \beta^2  \sigma^2 /2).
\end{eqnarray*}
For further analysis it would be convenient to consider 4 cases: 
\begin{enumerate}[(i)]
\item \(\sigma>1, a>0:\) in this case, \(\lambda_r(\beta) = \beta a +r \beta^2 \sigma^2 /2, \;
\gamma(\beta) = \beta a + \beta^2 \sigma^2 /2, \) and therefore  the rhs in \eqref{final} tends to zero iff 
\begin{eqnarray*}
H_r (\beta)
 =
 \bigl(
  1-r
\bigr)
\bigl(
  1 -r \frac{\beta^2 \sigma^2}{2}
\bigr)<0.
\end{eqnarray*}
There exists an \(r \in(1,2)\) such that this inequality is fulfilled iff \(\beta< \sqrt{2} / \sigma.\)
\item \(\sigma<1, a<0:\) in this case, \(\lambda_r(\beta) = r \beta^2  /2, \;
\gamma(\beta) =  \beta^2  /2, \) and therefore 
\begin{eqnarray*}
H_r (\beta) =  \bigl(
  1-r
\bigr)
\bigl(
  1 -r \frac{\beta^2}{2}
\bigr) < 0 
\end{eqnarray*}
 for some \(r \in (1,2)\) iff \(\beta< \sqrt{2}.\)
\item \(\sigma>1, a<0:\) it follows that 
\begin{eqnarray*}
\lambda_r (\beta) = 
\begin{cases}
r \beta^2 / 2, &\text{if \(\beta\leq \beta_\circ / r\)}, \\
\beta a + r \beta^2 \sigma^2 / 2		&\text{if \(\beta>\beta_\circ / r\),}
\end{cases}
\end{eqnarray*}
and 
\begin{eqnarray*}
\gamma (\beta) = 
\begin{cases}
 \beta^2 / 2, &\text{if \(\beta\leq \beta_\circ \)}, \\
\beta a + \beta^2 \sigma^2 / 2		&\text{if \(\beta>\beta_\circ \),}
\end{cases}
\end{eqnarray*}
Therefore, 
\begin{eqnarray}\label{Hr}
H_r (\beta) = 
\begin{cases}
\bigl(
  1-r
\bigr)
\bigl(
  1 -r \frac{\beta^2}{2}
\bigr), &\text{if \(\beta< \beta_\circ / r\)}, \\
\bigl(1-r \bigr) + r \bigl(
\beta a + \left(
   r \sigma^2 -1
\right)
\frac{\beta^2}{2}
\bigr),
&\text{if \(\beta_\circ / r<\beta< \beta_\circ\),}\\
 \bigl(
  1-r
\bigr)
\bigl(
  1 -r \frac{\beta^2 \sigma^2}{2}
\bigr),
&\text{if \(\beta> \beta_\circ\).}
\end{cases}
\end{eqnarray}
Taking into account the first and the third lines, 
we conclude that there exists some \(r \in (1,2)\) such that \(H_r (\beta) <0\) if \(\beta<\min(\sqrt{2}, \beta_\circ)\) or \(\beta_\circ < \beta< \sqrt{2}/\sigma.\) Therefore, the further analysis depends on the order of numbers \(\beta_\circ\) and \(\sqrt{2}/\sigma< \sqrt{2}\). More precisely, 

\begin{itemize}
\item  If \(a>\left(
  1-\sigma^2
\right) /( \sqrt{2} \sigma)\) (that is, \(\beta_\circ <\sqrt{2}/\sigma<\sqrt{2})\), then the LLN holds for  \(\beta < \sqrt{2}/\sigma.\)
\item  If \(\left(
  1-\sigma^2
\right) / \sqrt{2}<a<\left(
  1-\sigma^2
\right) /( \sqrt{2} \sigma)\) (that is, \(\sqrt{2}/\sigma<\beta_\circ <\sqrt{2})\), then the LLN holds for  \(\beta < \beta_\circ.\)
\item  If \(a<\left(
  1-\sigma^2
\right) / \sqrt{2}\) (that is, \(\beta_\circ >\sqrt{2})\), then the LLN holds for  \(\beta < \sqrt{2}.\) Note that in this case, \(a<\sqrt{2} (1-\sigma)\) and therefore the LLN is not fulfilled  for any \(\beta > \sqrt{2} \) due to the fact that \(P(\beta) = P^1 (\beta).\)
\end{itemize}
\item \(\sigma<1, a>0.\) The proof  follows the same lines as in the previous case. In particular, we get that 
\begin{eqnarray}\label{Hr}
H_r (\beta) = 
\begin{cases}
\bigl(
  1-r
\bigr)
\bigl(
  1 -r \frac{\beta^2 \sigma^2}{2}
\bigr), &\text{if \(\beta< \beta_\circ / r\)}, \\
\bigl(1-r \bigr) + r \bigl(
\frac{\beta^2}{2}
 - \beta a
 -
 \frac{\beta^2 \sigma^2}{2}
\bigr),
&\text{if \(\beta_\circ / r<\beta< \beta_\circ\),}\\
 \bigl(
  1-r
\bigr)
\bigl(
  1 -r \frac{\beta^2}{2}
\bigr),
&\text{if \(\beta> \beta_\circ\),}
\end{cases}
\end{eqnarray}
and the condition \(H_r(\beta)<0\) holds when \(\beta<\min(\sqrt{2}/ \sigma, \beta_\circ)\) or \(\beta_\circ < \beta< \sqrt{2}.\) To conclude the proof it is sufficient to consider 3 cases depending on the order of numbers \(\beta_\circ\) and \(\sqrt{2}<\sqrt{2}/\sigma.\)
\end{enumerate}
\subsection{Proof of Theorem~\ref{clt}}\label{sss}
Let us show 2 methods, which lead to the proof of this theorem. The first one is rather classical and is essentially based on the Lyapounov condition. The second proof  is based on the third part of  Proposition~3.1 from \cite{Panov2017}.

\textbf{Method 1. } Let us check that the Lyapounov condition holds (see (27.16) from \cite{Bill}): there exists \(\delta>0\) such that
\begin{eqnarray*}
\Omega_{n}:=
\frac{\MM_{n}(2+\delta)}{
	\left(
  \lfloor e^n \rfloor
\right)^{\delta/2} \left(
		\Var(e^{\beta \sqrt{n} Z_1})
	\right)^{1+\delta/2}
} \to 0, \qquad \mbox{as} \quad 
n \to \infty.
\end{eqnarray*}
Applying \eqref{exp} and \eqref{MM},
\begin{eqnarray*}
\Omega_n &\asymp&
\frac{
2^{2+\delta} \left(
\frac{
1
}{
2
}e^{((2+\delta)^2 \beta^2 / 2) n} + \frac{
1
}{
2
}\cdot e^{((2+\delta)\beta a + (2+\delta)^2 \beta^2  \sigma^2 /2) n }
\right)
}{
\left(
  \lfloor e^n \rfloor 
\right)^{\delta/2} \left(
\frac{
1
}{
2
} e^{ 2 \beta^2  n} + \frac{
1
}{
2
} \cdot e^{2 (\beta a + \beta^2  \sigma^2 ) n }
	\right)^{1+\delta/2}
}\\
&\asymp&
2^{2\Delta-1}
\exp\left\{
\left[
1 - \Delta + \Delta \left( g_\Delta (\beta) -
 h (\beta)
 \right)
\right] n 
 \right\},
\end{eqnarray*}
where \(\Delta = 1+\delta/2,\) and 
\begin{eqnarray*}
g_\Delta (\beta) &=& \max \left(
 2 \Delta \beta^2, \quad
  2 \beta a +  2\Delta \beta^2  \sigma^2 
\right), \\
h (\beta) &=& \max \left(
 2\beta^2, \quad
2\beta a + 2\beta^2  \sigma^2 
\right),
\end{eqnarray*}
There, one should find the condition on the  existence of \(\Delta>1\) such that 
\begin{eqnarray}\label{cc}
G_\Delta(\beta) := 1 - \Delta +2 \Delta \left( g_\Delta (\beta) -
 h (\beta)
 \right)
<0.
\end{eqnarray}
Note that this task was in fact previously considered on Step~4 of the proof of Theorem~\ref{lln}. The only difference is that \(\beta\) should be changed to \(2 \beta\) everywhere. This observation completes the proof.

\textbf{Method 2.}  Alternatively, let us show how the proof can be derived from the CLT for the summands. From the proof of Proposition~\ref{prop1}, we get the following 2 statements. 
\begin{enumerate}
\item For any sequence \(\gamma_n\), the distribution \(\mu_n\) of the random variable \(e^{\beta \sqrt{n} \xi_1} / \gamma_n\) satisfies
\begin{eqnarray*}
J_n(s, \tau) &:=& \lfloor e^n \rfloor  \int_0^{\tau} x^{s} \mu_{n} (dx) \\
&=& \frac{
\lfloor e^n \rfloor
e^{\beta^2 n s^2/2}
}{
\gamma_n^{s}
}
\Phi\left(
  \frac{
\log(\gamma_n \tau)
}{
\beta \sqrt{n}
}
-\beta s  \sqrt{n}
\right),
\end{eqnarray*}
where \(\tau>0\), \(s=\{1,2\}\). 
Assume now that 
 \begin{eqnarray}\label{RR}
\gamma_n = e^{n (1+ 2 \beta^2)/2} R(n)
\end{eqnarray}
 with some \(R(n)\). Then 
\begin{itemize}
\item  if \(R(n) \asymp c\) for some \(c>0,\) then \(J_n(2, \tau) \asymp c^{-s}\) for any \(\beta<\sqrt{2}/2;\)
\item if \(R(n) \to \infty\) as \(n \to \infty\) then  \(J_n (2, \tau) \to 0.\)
\end{itemize}
\item If the sequence \(\gamma_n\) is such that 
\begin{eqnarray}
\label{a1}
\lim_{n\to\infty}\frac{
\gamma_n
}{
e^{\sqrt{2} \beta n }
}
 = +\infty,
 \end{eqnarray} then 
\begin{eqnarray*}
\sum_{i=1}^{\lfloor e^n \rfloor} \P\left\{
e^{
    \beta \sqrt{n} \xi_i
}> \gamma_n x\right\} 
\to 0.
\end{eqnarray*}
\end{enumerate}
Similar outcomes are valid also for the second distribution:
\begin{enumerate}
\item for any sequence \(\gamma_n,\) the distribution \(\tilde\mu_n\) of  \(e^{ \beta (\sigma\sqrt{n} \xi_i + an)} / \gamma_n\) satisfies
\begin{eqnarray*}
\tilde{J}_n(s, \tau) &:=& \lfloor e^n \rfloor  \int_0^{\tau} x^{s} \tilde\mu_{n} (dx) \\
&=& \frac{
\lfloor e^n \rfloor
e^{\beta a n s + \beta^2 n s^2 \sigma^2 /2}
}{
\gamma_n^{s}
}
\Phi\left(
  \frac{
\log(\gamma_n \tau) - \beta a n 
}{
\beta \sigma \sqrt{n}
}
-\beta \sigma s  \sqrt{n}
\right),
\end{eqnarray*}
where \(\tau>0\), \(s=\{1,2\}\); in particular, from here it follows then if \[\gamma_n = e^{n (1+ 2 \sigma^2\beta^2 + 2 \beta a)/2} R(n)\] with some \(R(n),\) then 
\begin{itemize}
\item if \(R(n) \asymp c\) for some \(c>0,\) then \(\tilde{J}_n(2, \tau) \asymp c^{-s}\) when \(\beta< \sqrt{2}/2\);
\item if \(R(n) \to \infty,\) then 
 \(\tilde{J}_n (2, \tau) \to 0;\)
\end{itemize}

\item if the sequence \(\gamma_n\) is such that 
\begin{eqnarray}\label{a2} \lim_{n\to\infty}\frac{
\gamma_n
}{
e^{(\sqrt{2}  \sigma +a)\beta n}
}
 = +\infty,
 \end{eqnarray}
 then 
\begin{eqnarray*}
\sum_{i=1}^{\lfloor e^n \rfloor} \P\left\{
e^{
    \beta (\sigma\sqrt{n} \xi_i + an)
}> \gamma_n x\right\} 
\to 0.
\end{eqnarray*}
\end{enumerate}

Consider now the mixture of 2 distributions. Our aim is to find a normalizing sequence \(\breve\gamma_n\) such that both condition of  Proposition~3.1 from \cite{Panov2017} (part 3) are fulfilled. Natural candidate is \(\breve\gamma_n := \sqrt{\Var(\S_n(\beta))}\). Note that 
\begin{eqnarray*}
\breve\gamma_n = 
\left(
  \sum_{j=1}^{\lfloor e^n \rfloor}
  \Var 
e^{
    \beta \sqrt{n} Z_j
}
\right)^{1/2}
&=&
\left(
\lfloor e^n \rfloor \cdot 
\left(
  \E e^{2 \beta \sqrt{n} Z_1}
  -
  \left(
  \E e^{ \beta \sqrt{n} Z_1}  
\right)^2
\right)\right)^{1/2}\\
&\asymp&
\left(
\frac{
1
}{
2
}
\lfloor e^n \rfloor \cdot 
e^{2 \beta \max\{ \beta, a + \beta \sigma^2 \} n}\right)^{1/2},
\end{eqnarray*}
since for any \(c \in \R,\)\[
\E e^{c Z_1}  = \left(
e^{c^2/2} 
+
e^{c \sqrt{n} a + c^2 \sigma^2 /2} \right) /2.
\]
Therefore,  \(\breve\gamma_n\) has different asymptotics in the cases  when  \(\beta (1-\sigma^2)>a\) is fulfilled or not. In the first case,  \(\breve\gamma_n \asymp 2^{-1/2}e^{n(1 + 2\beta^2)/2} \). Under this choice of \(\breve\gamma_n\), \eqref{a1} trivially holds, and moreover from 
\begin{eqnarray*}
\lim_{n\to \infty}
\frac{
\gamma_n
}{
e^{(\sqrt{2}  \sigma +a)\beta n}
}
&\geq& 
2^{-1/2} 
\lim_{n\to \infty}
\frac{
e^{(1 + 2\beta a + 2\beta^2 \sigma^2)n /2}
}{
e^{(\sqrt{2}  \sigma +a)\beta n}
}\\
&=& 
2^{-1/2} 
\lim_{n\to \infty}
e^{n(1-\sqrt{2}\beta\sigma)^2/2} = +\infty,
\end{eqnarray*}
we conclude that \eqref{a2} also holds. Therefore,
\begin{eqnarray*}
\breve\Sigma_n 
&=&
\lfloor e^n \rfloor 
\cdot
\P \left\{ 
e^{\beta \sqrt{n} Z_1}
\geq \breve{\gamma}_n x
\right\}\\
&=&
\frac{
1
}{2
}
\lfloor e^n \rfloor 
\cdot
\P \left\{ 
e^{\beta \sqrt{n} \xi_1}
\geq \breve{\gamma}_nx
\right\}
+
\frac{
1
}{2
}
\lfloor e^n \rfloor 
\cdot
\P \left\{ 
e^{\beta \left(  \sigma\sqrt{n} \xi_1 + a n\right)}
\geq \breve{\gamma}_n x
\right\}
\to 0.
\end{eqnarray*}
On another hand, 
\begin{eqnarray}\label{JJ}
\lfloor e^n \rfloor  \int_0^{\tau} x^{s} m_{n} (dx)  
= 
\frac{
1
}{
2
}
J_n (s,\tau)
+
\frac{
1
}{
2
}
\tilde{J}_n (s,\tau),
\end{eqnarray}
where \(m_n\) stands for the distribution of  \(e^{\beta \sqrt{n} Z_1}/\gamma_n.\) Note that \(J_n(2,\tau) \asymp 2\) if \(\beta < \sqrt{2}/2\), and moreover \(\tilde{J}_n(2,\tau) \to 0\). Therefore,  we conclude that the sequence \(\S_n(\beta)/\breve\gamma_n\) converges to a standard  normal random variable, and 
\begin{eqnarray*}
a_n &=& \frac{
1
}{
2
}
J_n (1,1)
+
\frac{
1
}{
2
}
\tilde{J}_n (1,1)\\
&=&
\frac{
\lfloor e^n \rfloor
}{
\gamma_n
}
\left( 
\frac{1} {2} 
\E \left[e^{\beta s \sqrt{n} \xi_1}\right]
+
\frac{1} {2} 
\E \left[e^{\beta s (\sigma\sqrt{n} \xi_i + an)}\right]
\right)
=
\E[\S_n(\beta)].
\end{eqnarray*}
Finally, we conclude that the CLT holds  if \(\beta\) belongs to one of the following areas:
\begin{eqnarray*}
\BB_1 &:=& \Bigl\{\beta: \quad \beta (1-\sigma^2)>a, \;\;
\beta < \sqrt{2}/2 \Bigr\},\\
\BB_2 &:=& \Bigl\{\beta: \quad \beta (1-\sigma^2)<a, \;\;
\beta < \sqrt{2}/(2\sigma) \Bigr\}
\end{eqnarray*}
(the proof for \(\BB_2\) follows the same lines). Consideration of these areas depending on the relations between \(a\) and \(\sigma\) concludes the proof.

\subsection{Proof of Theorem~\ref{stable}}

Due to Proposition~\ref{prop1} (iii), we get that for \(\beta>\sqrt{2}/2,\) it holds
\begin{eqnarray*}
	\frac{\sum_{i=1}^{\lfloor e^n \rfloor}
e^{
    \beta \sqrt{n} \xi_i
}-\delta_n}{\gamma_n}  \eql 
	 F_{\alpha(\beta), \mu(\beta)}, \qquad n \to \infty.
\end{eqnarray*}
From this result, it follows that 
\begin{eqnarray*}
	\frac{\sum_{i=1}^{\lfloor e^n \rfloor}
e^{
    \beta \left( \sigma \sqrt{n} \xi_i + an\right)
}-\tilde\delta_n(\beta)}{\tilde\gamma_n (\beta)}  \eql 
	 F_{\alpha(\beta\sigma),\mu(\beta\sigma)}, \qquad n \to \infty,
\end{eqnarray*}
where \(\tilde\gamma_n (\beta)=e^{\beta a n} \gamma_n (\beta\sigma),\) \(\tilde\delta_n (\beta)=e^{\beta a n} \delta_n (\beta\sigma).\)

As for the mixtures, we should find some \(\breve\gamma_n=\breve\gamma_n(\beta)\) such that the sum 
\begin{eqnarray*}
\breve\Sigma_n 
&=&
\lfloor e^n \rfloor 
\cdot
\P \left\{ 
e^{\beta \sqrt{n} Z_1}
\geq \breve{\gamma}_n x
\right\}\\
&=&
\frac{
1
}{2
}
\lfloor e^n \rfloor 
\cdot
\P \left\{ 
e^{\beta \sqrt{n} \xi_1}
\geq \breve{\gamma}_nx
\right\}
+
\frac{
1
}{2
}
\lfloor e^n \rfloor 
\cdot
\P \left\{ 
e^{\beta \left(  \sigma\sqrt{n} \xi_1 + a n\right)}
\geq \breve{\gamma}_n x
\right\},
\end{eqnarray*}
converges to a non-trivial limit. Note that if \(\gamma_n(\beta) \lesssim \tilde{\gamma}_n(\beta)\) (equivalently, \(a>\sqrt{2}(1-\sigma)\)) then the choice \(\breve{\gamma}_n=\tilde\gamma_n\) yields the second summand in \(\breve{\Sigma}_n\) converges to \(\left( 2 \sqrt{2\pi}\right)^{-1}x^{-\sqrt{2}/(\beta\sigma)}\), while the first tends to 0:
\begin{eqnarray*}
\frac{
1
}{2
}
\lfloor e^n \rfloor 
\cdot
\P \left\{ 
e^{\beta \sqrt{n} \xi_1 }
\geq \breve\gamma_n x
\right\}
&=&
\frac{
1
}{2
}
\lfloor e^n \rfloor 
\cdot
\P \left\{ 
e^{\beta \sqrt{n} \xi_1 }
\geq \gamma_n \left( x
\frac{
\tilde\gamma_n
}{\gamma_n
}
\right)
\right\}\\
&\asymp& 
\frac{
1
}{2\sqrt{2 \pi} 
}
\left(x
\frac{
\tilde\gamma_n
}{\gamma_n
}
\right)^{-\sqrt{2} / \beta } 
\to 0, \quad \mbox{as} \; n \to \infty.
\end{eqnarray*}
Now let us consider formula \eqref{JJ}  with this choice of \(\breve\gamma_n.\) The asymptotic behaviour of the second summand follows from \eqref{Jn}, namely,
\begin{eqnarray*} \label{Jn}
 J_n (s, \tau) \asymp \begin{cases}
\frac{1}{
\sqrt{\pi} \left(\beta \sigma s - \sqrt{2}\right)
} \tau^{\beta \sigma s  - \sqrt{2}}, &\mbox{if \(\sqrt{2} - \beta \sigma s <0\),}\\
\exp\left\{n(\sqrt{2} - \beta \sigma s)^2 /2\right\}, &\mbox{if \(\sqrt{2} - \beta \sigma s >0\)},\\
\frac{
1
}{2
}
\exp\left\{n(\sqrt{2} - \beta \sigma s)^2 /2\right\}, &\mbox{if \(\sqrt{2} - \beta \sigma s =0\)},
\end{cases}
\end{eqnarray*}
and therefore \(\lim_{\tau \to 0} \limsup_{n \to \infty} J_n(2,\tau)=0\)  iff \(\beta> \sqrt{2} / (2\sigma)\).  
As for \(J_n(s,\tau),\) we get 
 \begin{eqnarray*}
J_n(s,\tau)
&=&
\frac{
\lfloor e^n \rfloor
e^{\beta^2 n s^2/2}
}{
\breve\gamma_n^{s}
}
\Phi\left(
  \frac{
\log(\breve\gamma_n \tau)
}{
\beta \sqrt{n}
}
-\beta s  \sqrt{n}
\right)
\end{eqnarray*}Taking into account that 
\(\log( \breve\gamma_n)  \asymp  (a+\sqrt{2}\sigma) \beta n\), and 
\begin{eqnarray*}
  \frac{
\log(\breve\gamma_n \tau)
}{
\beta \sqrt{n}
}
-\beta s  \sqrt{n}
= \left(
a+   \sqrt{2}\sigma - \beta s +
  \frac{
\log(\tau)
}{
\beta n
}
\right) \sqrt{n} - \frac{
\sigma
}{
2\sqrt{2}
}
\frac{
\log(2n)
}{
\sqrt{n}
},
\end{eqnarray*}
we conclude that 
\begin{eqnarray*}
 J_n (s, \tau) \asymp \begin{cases}
c_1 n^{r_1}
e^{r_2 n}
, &\mbox{if \(a+\sqrt{2}\sigma - \beta s <0\),}\\
\lfloor e^n \rfloor
\cdot
\E \left[e^{\beta s \sqrt{n} \xi_1}\right]
\breve\gamma_n^{-s}, &\mbox{if \(a+\sqrt{2}\sigma - \beta s >0\),}\\
\frac{
1
}{2
}
\lfloor e^n \rfloor
\cdot
\E \left[e^{\beta s \sqrt{n} \xi_1}\right]
\breve\gamma_n^{-s}, &\mbox{if \(a+\sqrt{2}\sigma - \beta s =0.\)}
\end{cases}
\end{eqnarray*}
with  \(r_2 =  1 - \left(a + \sqrt{2} \sigma\right)^2 /2<0\) and some \(r_1.\) Note that in the second case (\(a+\sqrt{2}\sigma - \beta s >0\)),
\begin{eqnarray*}
\log(J_n (2, \tau)) & \asymp & \left(
  2 \beta^2 - 2(a+\sqrt{2}\sigma) \beta +1
\right)n,
\end{eqnarray*}
where the quadratic form has 2 roots, \[\beta_\pm =
\frac{\left(
  a + \sqrt{2} \sigma
\right)
 \pm
\sqrt{
\left(
  a + \sqrt{2} \sigma
\right)^2 -2
}
}{2},
\]
where \(\beta_-=\beta_* /2\) and \(\beta_- \leq (a + \sqrt{2} \sigma) /2 \leq \beta_+,\)  and therefore \(J_n (2, \tau) \to 0 \) iff \(\beta > \beta_*/2.\)
Finally, we conclude that the convergence to a stable law holds if and only if \[\beta > \max \left\{ \frac{\beta_*}{2},  \frac{\sqrt{2}}{2 \sigma}. \right\}\]
In the considered area \(\left[
  a>\sqrt{2}(1-\sigma)
\right]\),
the first item is larger then the second if and only if 
\(\sigma>1\) and \(a< (1-\sigma^2)/ (\sqrt{2} \sigma).\)  This observation concludes the proof.

\appendix
\section{Previous research} 
\label{A}
As it was already mentioned in the introduction, the free energy of the considered system was recently studied by Grabchak, Molchanov \cite{GM17}. In this appendix, we shortly discuss the main results from that research.

Free energy of the system with \(\lfloor e^n \rfloor\) standard normal energy levels is equal  to 
\begin{eqnarray*}
P_1 (\beta) = 
 \lim_{n \to \infty} \frac{
\ln S_n^1 (\beta) 
}{
n
}
&=&
\left\{
\begin{aligned}
1+ \beta^2 /2, & \qquad \beta \leq \sqrt{2}, \\
\sqrt{2} \beta, & \qquad \beta >\sqrt{2}.
\end{aligned}
\right.,
\end{eqnarray*}
whereas the free energy of system with  \(\lfloor e^n \rfloor\)  energy levels having  \(\NN(\sqrt{n}a,\sigma^2)\) distribution equals
\begin{eqnarray*}
P_2 (\beta) = 
 \lim_{n \to \infty} \frac{
\ln S_n^2 (\beta) 
}{
n
}
&=&
\left\{
\begin{aligned}
1+ \beta a + \beta^2 \sigma^2 /2, & \qquad \beta \leq \sqrt{2}/\sigma \\
\left( \sigma \sqrt{2} + a\right) \beta, & \qquad \beta >\sqrt{2}/ \sigma.
\end{aligned}
\right.
\end{eqnarray*}
It turns out that  the free energy of the system with mixture distribution of energy levels is equal to \(P(\beta)=\max\{P_1(\beta), P_2(\beta)\}\). This observation leads to the conclusion, that \(P(\beta)\) essentially differs between 6 possible types of relations between \(a\) and \(\sigma\),
see Figure~\ref{fig2}:
\begin{figure}
\begin{center}
\includegraphics[width=0.8\linewidth ]{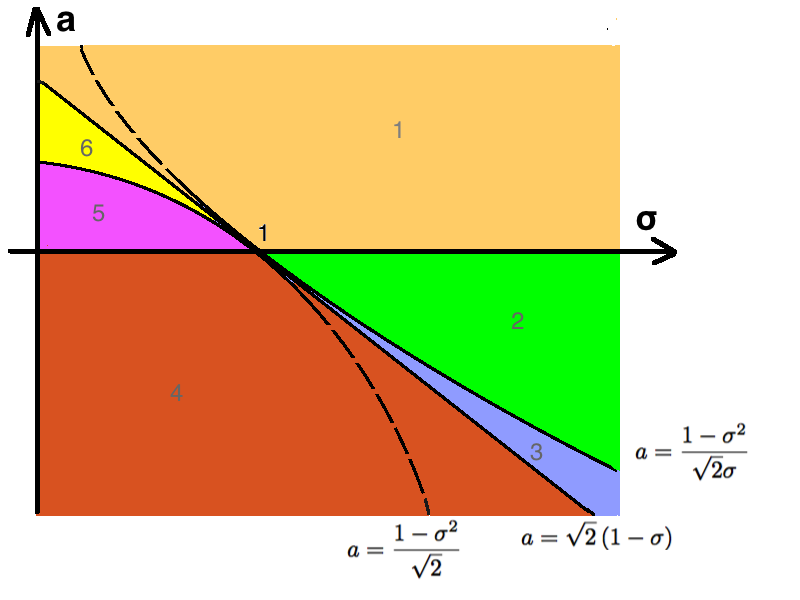}\caption{\label{fig2}Illustration of the  asymptotic behavior of the free energy depending on the parameters \(a\) and \(\sigma.\)}
\end{center}\end{figure}
\begin{enumerate}
\item Zone 1 (orange), \(a>0\) and \(a> \sqrt{2} ( 1- \sigma)\): \quad \(P(\beta) = P_2(\beta)\).
\item Zone 2 (green), \(a<0\) and \(a>(1-\sigma^2) / (\sqrt{2} \sigma)\):
\begin{eqnarray*}
P(\beta) = \left\{
\begin{aligned}
P_1 (\beta), & \qquad \beta \leq \beta_{\circ} \\
P_2 (\beta), & \qquad \beta >\beta_{\circ}.
\end{aligned}
\right.
\end{eqnarray*}
with \(\beta_{\circ} = 2a / (1-\sigma^2)\).
\item Zone 3 (blue), \(a<0\) and \(\sqrt{2} ( 1- \sigma)<a<(1-\sigma^2) / (\sqrt{2} \sigma)\): 
\begin{eqnarray*}
P(\beta) = \left\{
\begin{aligned}
P_1 (\beta), & \qquad \beta \leq \beta_{*} \\
P_2 (\beta), & \qquad \beta >\beta_{*}.
\end{aligned}
\right.
\end{eqnarray*}
with \(\beta_{*} = \left(
  \sigma \sqrt{2} +a
\right) - \sqrt{\left(
  \sigma \sqrt{2} +a
\right)^2 - 2}\).
\item Zone 4 (red), \(a<0\) and \(a< \sqrt{2} ( 1- \sigma)\): \(P(\beta) = P_2(\beta).\)
\item Zone 5 (purple), \(a>0\) and \((1-\sigma^2) / \sqrt{2}<a<\sqrt{2} ( 1- \sigma)\): 
\begin{eqnarray*}
P(\beta) = \left\{
\begin{aligned}
1 + a \beta + \sigma^2 \beta^2 /2, & \qquad \beta \leq \beta_{\diamond}, \\
\sqrt{2} \beta, & \qquad  \beta > \beta_\diamond,
\end{aligned}
\right.
\end{eqnarray*}
where \(\beta_\diamond = \left(
 ( \sqrt{2} -a ) - \sqrt{\left(
  \sqrt{2} - a 
\right)^2  - 2 \sigma^2}
\right)/\sigma^2.\)
\item Zone 6 (yellow): \(a>0\) and \(a<(1-\sigma^2) / \sqrt{2}\): 
\begin{eqnarray*}
P(\beta) = \left\{
\begin{aligned}
1 + a \beta + \sigma^2 \beta^2 /2, & \qquad \beta \leq \beta_{\circ}, \\
1 + \beta^2 /2, & \qquad \beta_\circ< \beta < \sqrt{2}, \\ 
\sqrt{2} \beta, & \qquad \beta > \sqrt{2}.
\end{aligned}
\right.
\end{eqnarray*}
\end{enumerate}
It would be a worth mentioning that these zones do not completely coincide with the zones on Figure~\ref{fig1}. Nevertheless, the critical values \(\beta_*\) and \(\beta_\diamond\) appears under the same assumptions on the relation between \(a\) and \(\sigma\).

\bibliographystyle{plain}
\bibliography{Panov_bibliography}

\end{document}